\documentclass[11pt,a4paper]{article}

\usepackage{amsmath}
\usepackage{amsthm}
\usepackage{amssymb}
\usepackage{graphicx}
\usepackage{cain_arxiv}

\theoremstyle{definition}
\newtheorem{definition}{Definition}[section]

\newtheorem{example}[definition]{Example}

\newtheorem{remark}[definition]{Remark}
\newtheorem{pumprule}{Pumping rule}
\newtheorem{shiftrule}{Shift rule}

\theoremstyle{plain}
\newtheorem{proposition}[definition]{Proposition}
\newtheorem{lemma}[definition]{Lemma}
\newtheorem{theorem}[definition]{Theorem}

\numberwithin{equation}{section}

\def\fullref#1#2{%
  \ifdefined\hyperref%
    {\hyperref[#2]{#1 \penalty 200\relax\ref*{#2}}}%
  \else%
    {#1 \penalty 200\relax\ref{#2}}%
  \fi%
}

\newcommand{\defterm}[1]{\textit{#1}}

\newcommand{\fsa}[1]{\mathfrak{#1}}
\newcommand{\tm}[1]{\mathfrak{#1}}

\newcommand{\nset}{\mathbb{N}}
\newcommand{\zset}{\mathbb{Z}}

\newcommand{\emptyword}{\varepsilon}

\newcommand{\conv}{\mathrm{conv}}

\newcommand{\inpargraphic}[1]{\par\medskip\centerline{\includegraphics{#1}}\medskip\noindent\ignorespaces}

\begin{document}

\title{Subalgebras of FA-pre\-sent\-a\-ble algebras}
\author{Alan J. Cain \& Nik Ru\v{s}kuc}
\date{}

\thanks{The first author's research was funded by the European
  Regional Development Fund through the programme {\sc COMPETE} and by
  the Portuguese Government through the {\sc FCT} (Funda\c{c}\~{a}o
  para a Ci\^{e}ncia e a Tecnologia) under the project {\sc
    PEst-C}/{\sc MAT}/{\sc UI0}144/2011 and through an {\sc FCT}
  Ci\^{e}ncia 2008 fellowship. Two visits by the first author to the
  University of St Andrews, where much of the research described in
  this paper was carried out, were supported by the {\sc EPSRC}-funded
  project {\sc EP}/{\sc H0}11978/1 `Automata, Languages, Decidability
  in Algebra'. The authors thank Richard M. Thomas for helpful discussions.}

\maketitle

\address[AJC]{%
Centro de Matem\'{a}tica, Faculdade de Ci\^{e}ncias, Universidade do Porto, \\
Rua do Campo Alegre 687, 4169--007 Porto, Portugal
}
\email{%
ajcain@fc.up.pt
}
\webpage{%
www.fc.up.pt/pessoas/ajcain/
}

\address[NR]{%
School of Mathematics and Statistics, University of St Andrews, \\
North Haugh, St Andrews, Fife KY16 9SS, United Kingdom
}
\email{%
nik@mcs.st-andrews.ac.uk
}
\webpage{%
turnbull.mcs.st-and.ac.uk/~nik/
}

\begin{abstract}
Automatic presentations, also called FA-pre\-sent\-a\-tions, were
introduced to extend finite model theory to infinite structures whilst
retaining the solubility of fundamental decision problems. This paper
studies FA-pre\-sent\-a\-ble algebras. First, an example is given to
show that the class of finitely generated FA-pre\-sent\-a\-ble
algebras is not closed under forming finitely generated subalgebras,
even within the class of algebras with only unary operations. However,
it is proven that a finitely generated subalgebra of an FA-pre\-sent\-a\-ble
algebra with a single unary operation is itself
FA-pre\-sent\-a\-ble. Furthermore, it is proven that the class of unary
FA-pre\-sent\-a\-ble algebras is closed under forming finitely
generated subalgebras, and that the membership problem for such
subalgebras is decidable.
\end{abstract}

\section{Introduction}

Automatic presentations, also known as FA-pre\-sent\-a\-tions, were
introduced by Khoussainov \& Nerode \cite{khoussainov_autopres} to
fulfill a need to extend finite model theory to infinite structures
while retaining the solubility of interesting decision
problems. Informally, an FA-pre\-sent\-a\-tion for a relational
structure consists of a regular language of abstract representatives
for elements of the structure such that the relations of the structure
can be recognized by synchronous finite
automata. FA-pre\-sent\-a\-tions have been considered for structures
such as orders
\cite{khoussainov_autopo,khoussainov_linearorders,delhomme_autopresordinal},
graphs \cite{khoussainov_unarygraphalgorithmic}, and groups,
semigroups, and rings
\cite{oliver_autopresgroups,cort_apsg,nies_rings}.

This paper studies subalgebras of FA-pre\-sent\-a\-ble algebras. In
the particular case of groups, it was already known that there exists
an FA-pre\-sent\-a\-tion for the group $\zset \times \zset$ under
which the sublanguage of representatives for elements of any
non-trivial cyclic subgroups is not regular
\cite[\S~6]{nies_autopresabel}. However, such subgroups, like all
abelian groups \cite[Theorem~3]{oliver_autopresgroups}. are
FA-pre\-sent\-a\-ble with a different language of
representatives. We construct an example of a finitely generated
FA-pre\-sent\-a\-ble algebra that contains a non-FA-pre\-sent\-a\-ble
finitely generated subalgebra
(\fullref{Example}{ex:nonfasubalgebra}). This shows that the class of
FA-pre\-sent\-a\-ble algebras is not closed under taking finitely
generated subalgebras. Furthermore, this non-closure holds even within
the class of algebras equipped with only unary operations
(\fullref{Remark}{rem:evenunary}). However, the class of
FA-pre\-sent\-a\-ble algebras with a single unary operation is closed
under forming finitely generated subalgebras
(\fullref{Proposition}{prop:oneunaryop}).

On the other hand, we prove that the class of algebras that admit
unary FA-pre\-sent\-a\-tions (that is, FA-pre\-sent\-a\-tions over a
one-letter alphabet) \emph{is} closed under forming finitely generated
subalgebras (\fullref{Theorem}{thm:subalgebras}). The proof depends on
the sublanguage of representatives for elements of the subalgebra
being regular and effectively constructible
(\fullref{Theorem}{thm:subalgebralang}), which also implies that the
membership problem is decidable for such subalgebras
(\fullref{Theorem}{thm:membership}). We also prove that finitely
generated unary FA-pre\-sent\-a\-ble algebras have growth level
bounded by a linear function (\fullref{Proposition}{prop:unaryalggrowth}).
%% This contrast between the unary and general case
%% parallels the relative successes of the programme of characterizing
%% unary FA-pre\-sent\-a\-ble structures (see, for example,
%% \cite[Ch.~7]{blumensath_diploma}) versus that of characterizing
%% general FA-pre\-sent\-a\-ble structures (see, for example, the survey
%% \cite{rubin_survey}).

These results for unary FA-pre\-sent\-a\-tions are proved using a new
diagrammatic representation, developed in
\fullref{\S}{sec:pumping}. This representation allows us to visualize
and manipulate elements of a unary FA-pre\-sent\-a\-ble relational
structure in a way that is more accessible than the corresponding
arguments using languages and automata. In a forthcoming paper
\cite{cr_binrel}, we deploy this representation in an analysis of
unary FA-presentable binary relations. This representation is thus
potentially a unifying framework in which to reason about unary
FA-pre\-sent\-a\-ble algebraic and relational structures.

\section{Preliminaries}

The reader is assumed to be familiar with the theory of finite
automata and regular languages; see \cite[Chs~2--3]{hopcroft_automata}
for background reading. The empty word (over any alphabet) is denoted
$\emptyword$.

\begin{definition}
Let $L$ be a regular language over a finite alphabet $A$. Define, for $n \in \nset$,
\[
L^n = \{(w_1,\ldots,w_n) : w_i \in L\text{ for $i=1,\ldots,n$}\}.
\]
Let $\$$ be a new symbol not in $A$. The mapping $\conv : (A^*)^n\to ((A\cup\{\$\})^n)^*$ is defined as follows. Suppose
\begin{align*}
w_1 &= w_{1,1}w_{1,2}\cdots w_{1,m_1},\\
w_2 &= w_{2,1}w_{2,2}\cdots w_{2,m_2},\\
&\vdots\\
w_n &= w_{n,1}w_{n,2}\cdots w_{n,m_n},
\end{align*}
where $w_{i,j} \in A$. Then $\conv(w_1,\ldots,w_n)$ is defined to be
\[
(w_{1,1},w_{2,1},\ldots,w_{n,1})(w_{1,2},w_{2,2},\ldots,w_{n,2})\cdots (w_{1,m},w_{2,m},\ldots,w_{n,m}),
\]
where $m = \max\{m_i:i=1,\ldots,n\}$ and with $w_{i,j} = \$$ whenever
$j > m_i$.
\end{definition}

Observe that the mapping $\conv$ maps an $n$-tuple of words to a word of $n$-tuples.

\begin{definition}
Let $A$ be a finite alphabet, and let $R \subseteq (A^*)^n$ be a
relation on $A^*$. Then the relation $R$ is said to be \defterm{regular} if
\[
\conv R = \{\conv(w_1,\ldots,w_n) : (w_1,\ldots,w_n) \in R\}
\]
is a regular language over $(A\cup\{\$\})^n$.
\end{definition}

\begin{definition}
Let $\mathcal{S}=(S,R_1,\ldots,R_n)$ be a relational structure. Let
$L$ be a regular language over a finite alphabet $A$, and let $\phi :
L\rightarrow S$ be a surjective mapping. Then $(L,\phi)$ is an
\defterm{automatic presentation} or an \defterm{FA-pre\-sent\-a\-tion}
for $\mathcal{S}$ if, for all relations $R \in \{=,R_1,\ldots,R_n\}$
of relation
\[
\Lambda(R,\phi)=\{(w_1,w_2,\ldots,w_{r_i})\in L^{r_i}:R(w_1\phi,\ldots,w_{r}\phi)\},
\]
where $r$ is the arity of $R$, is regular.

If $\mathcal{S}$ admits an FA-pre\-sent\-a\-tion, it is said to be
\defterm{FA-pre\-sent\-a\-ble}.

If $(L,\phi)$ is an FA-pre\-sent\-a\-tion for $\mathcal{S}$ and the
mapping $\phi$ is injective (so that every element of the structure
has exactly one representative in $L$), then $(L,\phi)$ is said to be
\defterm{injective}.

If $(L,\phi)$ is an FA-pre\-sent\-a\-tion for $\mathcal{S}$ and $L$ is
a language over a one-letter alphabet, then $(L,\phi)$ is a
\defterm{unary} FA-pre\-sent\-a\-tion for $\mathcal{S}$, and
$\mathcal{S}$ is said to be \defterm{unary FA-pre\-sent\-a\-ble}.
\end{definition}

Every FA-pre\-sent\-a\-ble structure admits an injective \defterm{binary}
FA-pre\-sent\-a\-tion; that is, where the language of representatives is
over a two-letter alphabet; see
\cite[Corollary~4.3]{khoussainov_autopres} and
\cite[Lemma~3.3]{blumensath_diploma}. Therefore the class of binary
FA-pre\-sent\-a\-ble structures is simply the class of FA-pre\-sent\-a\-ble
structures. However, there are many structures that admit
FA-pre\-sent\-a\-tions but not unary FA-pre\-sent\-a\-tions: for instance, any
finitely generated virtually abelian group is FA-pre\-sent\-a\-ble
\cite[Theorem~8]{oliver_autopresgroups}, but unary FA-pre\-sent\-a\-ble
groups must be finite \cite[Theorem~7.19]{blumensath_diploma}. Thus
there is a fundamental difference between unary FA-pre\-sent\-a\-ble
structures and all other FA-pre\-sent\-a\-ble
structures.

The fact that a tuple of elements $(s_1,\ldots,s_n)$ of a structure $\mathcal{S}$
satisfies a first-order formula $\theta(x_1,\ldots,x_n)$ is denoted $\mathcal{S}
\models \theta(s_1,\ldots,s_n)$. 

\begin{proposition}[{\cite[Theorem~4.4]{khoussainov_autopres}}]
\label{prop:firstorderreg}
Let $\mathcal{S}$ be a structure with an FA-pre\-sent\-a\-tion $(L,\phi)$. For
every first-order formula $\theta(x_1,\ldots,x_n)$ over the structure, the relation
\[
\Lambda(\theta,\phi) = \bigl\{(w_1,\ldots,w_n) \in L^n : \mathcal{S} \models\theta(w_1\phi,\ldots,w_n\phi)\bigr\}
\]
is regular, and an automaton recognizing it can be effectively
constructed.
\end{proposition}

\fullref{Proposition}{prop:firstorderreg} is fundamental to the theory
of FA-pre\-sent\-a\-tions and will be used without explicit reference
throughout the paper.

The following important result shows that in the case of unary
FA-pre\-sent\-a\-tions for infinite structures, we can assume that the
language of representatives is the language of \emph{all} words over a
one letter alphabet:

\begin{theorem}[{\cite[Theorem~3.1]{crt_unaryfa}}]
\label{thm:allwords}
Let $\mathcal{S}$ be an infinite relational structure that admits a unary
FA-pre\-sent\-a\-tion. Then
$\mathcal{S}$ has an injective unary FA-pre\-sent\-a\-tion
$(a^*,\psi)$.
\end{theorem}

\section{Subalgebras of FA-pre\-sent\-a\-ble algebras}
\label{sec:fapalgebras}

This section presents various contrasting results for finitely
generated subalgebras of FA-pre\-sent\-a\-ble subalgebras. First, a
positive result is proven: the class of FA-pre\-sent\-a\-ble algebras
whose signature consists of a single unary operation is closed under
forming finitely generated subalgebras
(\fullref{Proposition}{prop:oneunaryop}). Then we exhibit an example
of a finitely generated FA-pre\-sent\-a\-ble algebra that admits a
non-FA-pre\-sent\-a\-ble finitely generated subalgebra
(\fullref{Example}{ex:nonfasubalgebra}). Although this example algebra
is equipped with a binary operation, we note afterwards how it can be
modified into an algebra with only unary operations
(\fullref{Remark}{rem:evenunary}), thus contrasting our previous
positive result.

Although this section shows that the class of FA-pre\-sent\-a\-ble
algebras is not closed under forming finitely generated
subalgebras. However, closure under forming finitely generated
subalgebras may hold within classes of FA-pre\-sent\-a\-ble algebras of
a particular type. For instance, the following result holds:

\begin{proposition}
Every finitely generated subgroup of an FA-pre\-sent\-a\-ble group is
FA-presentable.
\end{proposition}

\begin{proof}
A finitely generated subgroup of an FA-presentable group is virtually
abelian by \cite[Theorem~10(i)]{nies_rings} and hence
FA-pre\-sent\-a\-ble by \cite[Theorem~3]{oliver_autopresgroups}.
\end{proof}

\begin{proposition}
\label{prop:oneunaryop}
Let $\mathcal{S} = (S,\sigma)$ be an algebra, where $\sigma$ is a
unary operation. Then every finitely generated subalgebra of
$\mathcal{S}$ is FA-pre\-sent\-a\-ble.
\end{proposition}

\begin{proof}
Let $X = \{x_1,\ldots,x_n\}$ be a finite subset of $S$ and let
$\mathcal{T}$ be the subalgebra of $\mathcal{S}$ that $X$
generates. We will inductively define FA-presentations
$(L_i,\phi_i)$ for the subalgebras $\mathcal{T}_i$ of $\mathcal{S}$
generated by $X_i = \{x_1,\ldots,x_i\}$. To avoid having to treat the
case $i=1$ separately, formally define $L_0$ to be the empty language,
$\phi_0$ to be the empty map and $\Lambda(\sigma,\phi_0)$ to be the
empty relation.

So suppose we have an FA-presentation $(L_{i-1},\phi_{i-1})$ for the
subalgebra $\mathcal{T}_i$ generated by $X_{i-1}$, where $i \in
\{1,\ldots,n\}$. Notice that the subalgebra generated by $X_i$
consists of elements in the set $L_{i-1}\phi_{i-1} \cup \{x_i\sigma^k : k
\in \nset^0\}$. There are three cases to consider:
\begin{enumerate}

\item For some $k \in \nset^0$, the element $x_i\sigma^k$ lies in
  $L_{i-1}\phi_{i-1}$, with $u\phi_{i-1} = x_i\sigma^k$. In this case, let
\[
L_i = L_{i-1} \cup \bigl\{p_{i,j} : j \in \{0,\ldots,k-1\}\bigr\};
\]
then $L_i$ is regular. Define $\phi_i$ by $\phi_i|_{L_{i-1}} =
\phi_{i-1}$ and $p_{i,j}\phi_i = x_i\sigma^j$ for $j \in
\{0,\ldots,k-1\}$. Then
\begin{align*}
\Lambda(\sigma,\phi_i) ={}& \Lambda(L_{i-1},\phi_{i-1}) \\
&\cup \bigl\{(p_{i,j},p_{i,j+1}) : j \in \{0,\ldots,k-2\}\bigr\} \cup \{(p_{i,k-1},u)\}
\end{align*}
is a union of regular relations and is thus regular.

\item The element $x_i\sigma^k$ does not lie in $L_{i-1}\phi_{i-1}$
  for any $k \in \nset^0$, but $x_i\sigma^k = x_i\sigma^{k+m}$ for
  some $k \in \nset^0$ and $m \in \nset$. In this case, let
\[
L_i = L_{i-1} \cup \bigl\{p_{i,j} : j \in \{0,\ldots,k+m-1\}\bigr\};
\]
then $L_i$ is regular. Define $\phi_i$ by $\phi_i|_{L_{i-1}} =
\phi_{i-1}$ and $p_{i,j}\phi_i = x_i\sigma^j$ for $j \in \{0,\ldots,k+m-1\}$. Then
\begin{align*}
\Lambda(\sigma,\phi_i) ={}& \Lambda(L_{i-1},\phi_{i-1}) \\
&\cup \bigl\{(p_{i,j},p_{i,j+1}) : j \in \{0,\ldots,k+m-2\}\bigr\} \cup \{(p_{i,k+m-1},p_{i,k})\}
\end{align*}
is a union of regular relations and is thus regular.

\item The element $x_i\sigma^k$ does not lie in $L_{i-1}\phi_{i-1}$
  for any $k \in \nset^0$, and all $x_i\sigma^k$ are distinct. In this case, let
\[
L_i = L_{i-1} \cup ia^*,
\]
where $i$ is treated as a symbol $a$ is a new symbol; then $L_i$ is
regular. Define $\phi_i$ by $\phi_i|_{L_{i-1}} = \phi_{i-1}$ and
$ia^j\phi_i = x_i\sigma^j$ for $j \in \nset^0$. Then
\[
\Lambda(\sigma,\phi_i) = \Lambda(L_{i-1},\phi_{i-1}) \cup \{(ia^j,ia^{j+1}) : j \in \nset^0\}
\]
is a union of regular relations and is thus regular.

\end{enumerate}
In any case, $(L_{i+1},\phi_{i+1})$ is an FA-presentation for the
subalgebra $\mathcal{T}_i$. In particular, $\mathcal{T} =
\mathcal{T}_n$ is FA-presentable.
\end{proof}

Note, however, that the proof of
\fullref{Proposition}{prop:oneunaryop} does not yield an effective
algorithm for constructing the subalgebra. If such an algorithm
existed, reachability in the configuration graphs of deterministic
Turing machines would be soluble.

\begin{example}
\label{ex:nonfasubalgebra}
The example algebra $\mathcal{X}$ will consist of the disjoint union of a
semilattice and two copies of the configuration graph of a Turing
machine, augmented by extra unary operations.

\begin{figure}[tb]
\centerline{\includegraphics{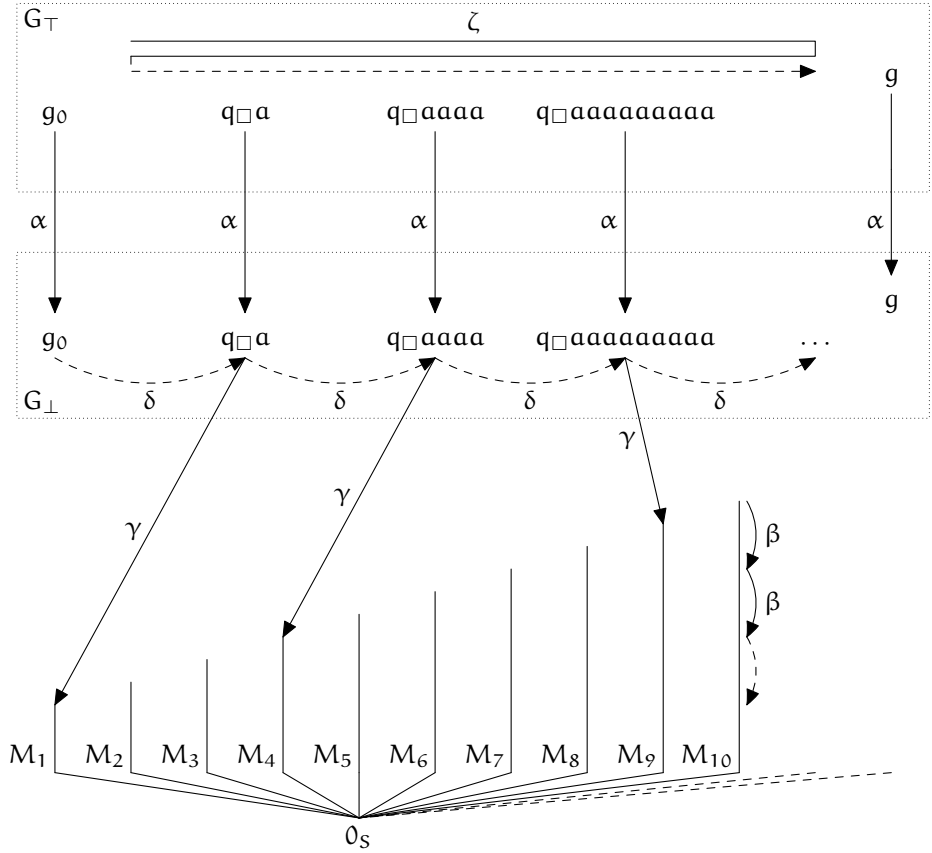}}
\caption{Schematic diagram of the algebra $\mathcal{X}$. The unary
  operation $\alpha$ maps each configuration of $G_\top$ to the
  corresponding configuration of $G_\bot$; the operation $\beta$ maps
  `down' within the semilattice; the operation $\gamma$ maps from
  `$q_\square$' configurations of $G_\bot$ to the maximum element of
  $M_i$, where $i$ is the number of symbols $a$ on the tape; the
  operation $\delta$ mimics the computation step of the Turing machine
  $\tm{T}$; and the operation $\zeta$ iterates through $G_\top$ in
  length-plus-lexicographic order.}
\label{fig:nonfasubalgebra1}
\end{figure}

For each $i \in \nset$, let $M_i$ be a chain of $2^i$ elements. Let
$S$ be the zero-union of all the $M_i$; the zero of $S$ is denoted
$0_S$. Notice that $S$ is a semilattice and can either be viewed as a
partially-ordered set or a semigroup where the multiplication is the
meet operation.

Let $\tm{T}$ be a deterministic Turing machine that generates
sequences of symbols $a^{j^2}$, where $j \in \nset$. More precisely,
$\tm{T}$ starts with an empty tape, performs some computation and
arrives in a distinguished state $q_\square$ with its tape contents
being $a^{1^2}$, then computes again and reaches state $q_\square$
with its contents being $a^{2^2}$. In general at various points during
its computation $\tm{T}$ has tape contents $a^{j^2}$ for every $j \in
\nset$, and $\tm{T}$ enters state $q_\square$ exactly when its tape
contents are $a^{j^2}$ for some $j \in \nset$. Notice that $\tm{T}$
runs forever without halting. Suppose $Q$ is the state set and $B$ the
tape alphabet of $\tm{T}$.

Recall that an \defterm{instantaneous description}, or
\defterm{configuration}, of $\tm{T}$ consists of its state, its tape
contents, and the position of its read/write head on its tape. The
\defterm{configuration} graph of $\tm{T}$ is an infinite graph whose
vertices are all conceivable configurations of $\tm{T}$, with a
directed edge from $g$ to $g'$ precisely if $\tm{T}$, when in
configuration $g$, can make a single computation step and reach
configuration $g'$. Note that in general not all configurations are
reachable from the initial configuration.

Let $G_\top$ and $G_\bot$ be two copies of the configuration graph of
$\tm{T}$. The carrier set for the algebra $\mathcal{X}$ will be $X = S
\cup G_\top \cup G_\bot$. The semilattice $S$ is already equipped with
a multiplication $\circ$; extend this multiplication to $X$ by
defining $g \circ g' = g$ and $g \circ s = s \circ g = g$ for $g,g'
\in G_\top \cup G_\bot$ and $s \in S$. The configuration graph
$G_\bot$ is equipped with a directed edge relation $\delta$. Since
$\tm{T}$ is deterministic, each vertex of the graph has outdegree $1$,
and so the relation $\delta$ can be viewed as a unary
operation. Extend $\delta$ to $X$ by $x\delta = x$ for all $x \in S
\cup G_\top$. We emphasize that $\delta$ acts like a computation
step by $\tm{T}$ in the configuration graph $G_\bot$, but acts like
the identity map on the configuration graph $G_\top$.

Now define three new unary operations. First, $\alpha$ sends each
configuration in $G_\top$ to the corresponding configuration in
$G_\bot$, and otherwise (for all elements of $S \cup G_\bot$) acts
like the identity map. Second, $\beta$ sends each element of a chain
$M_i$ to the element immediately below it in that same $M_i$, sends
the minimum element of each $M_i$ to $0_S$, and otherwise (for all
elements of $\{0_S\} \cup G_\top \cup G_\bot$) acts like the identity
map. Third, $\gamma$ maps configurations in $G_\bot$ with state
$q_\square$ and tape contents $a^k$ to the maximum element of the
chain $M_k$, and otherwise (for all other elements of $G_\bot$ and all
elements of $S \cup G_\top$) acts like the identity map. We will shortly
define yet another unary operation $\zeta$, but we must first set up
the FA-pre\-sent\-a\-tion. The algebra $\mathcal{X}$ will be
$(X,\circ,\alpha,\beta,\gamma,\delta,\zeta)$.

Let $L$ be the language $\{z\} \cup \{0,1\}^* \cup
\{\top,\bot\}B^*QB^*$, where $z$ is a new symbol not in $B$ or $Q$. Define $\phi : L \to X$ as follows:
\begin{itemize}
\item $z\phi = 0_S$.
\item If $u \in \{0,1\}^k$, then $u\phi$ is the $u$-th element
  (interpreting $u$ as a binary number) from the bottom in
  $M_{k}$. (Notice that since $M_{k}$ contains exactly $2^k$ elements,
  $\phi$ restricts to a bijection between $\{0,1\}^k$ and $M_k$.)
\item If $t \in \{\top,\bot\}$, $u,v \in B^*$ and $q \in Q$, then
  $(tuqv)\phi$ is the configuration in $G_t$ where the
  state is $q$, the tape contains $uv$, and the head points to the
  first symbol in $v$.
\end{itemize}
Let us first show that the definition of FA-pre\-sent\-a\-bility is
satisfied for the operations $\circ$, $\alpha$, $\beta$, $\gamma$, and
$\delta$.

To see that $\Lambda(\circ,\phi)$ is regular, it is simplest to notice
that the $\Lambda(\leq,\phi)$ is regular, where $\leq$ is the order on
the semilattice $S$, since an automaton recognizing
$\Lambda(\leq,\phi)$ must simply compare the lengths of two strings
over $\{0,1\}^*$ and compare them as binary numbers, and also always
accept if the left-hand input word is $z$ and the right lies in $\{z\}
\cup \{0,1\}^*$. Then, since $\circ$ is first-order definable in terms
of $\leq$, it follows that $\Lambda(\circ,\phi)$ is regular.

Next,
\begin{align*}
\Lambda(\alpha,\phi) ={}& \{(\top uqv,\bot uqv) : u,v \in B^*, q
\in Q\} \\
&\cup \{(w,w) : w \in \{z\} \cup \{0,1\}^* \cup \bot B^*QB^*\}
\end{align*}
is clearly regular.

An automaton recognizing $\Lambda(\beta,\phi)$ need only
decrement a binary number by $1$, recognize $(0^k,z)$, and
recognize the identity relation on $\{z\} \cup \{\top,\bot\}B^*QB^*$.

Now,
\[
\Lambda(\gamma,\phi) = \{(\bot a^kqa^l,1^{k+l}) : k,l \in \nset^0\} \cup \{(u,u) : u \in L - \bot a^*Qa^*\},
\]
which is easily seen to be regular.

The relation $\Lambda(\delta,\phi)$ is easily seen to be regular,
since each computation of a Turing machine makes only a small
localized change to the configuration as represented by words in
$B^*QB^*$; see \cite[p.~374]{khoussainov_autopres}.

We can now define our last operation $\zeta$. Let $\sqsubseteq$ be the
length-plus-lex\-i\-co\-graph\-ic ordering of words in $\top B^*QB^*$ induced
by some order on $\{\top\} \cup B \cup Q$. For any element $g \in
G_\top$, define $g\zeta$ as follows. Let $u$ be the unique word in
$\top B^*QB^*$ with $u\phi = g$. Let $u'$ be the word in $\top
B^*QB^*$ that succeeds $u$ in the $\sqsubset$ ordering. Then $g\zeta$ is
defined to be $u'\phi$. For all $x \in S \cup G_\bot$, define $x\zeta =
x$. Notice that
\begin{align*}
\Lambda(\zeta,\phi) ={}& \bigl\{(u,u') : u \in \top B^*QB^* \land (u \sqsubset u') \\
&\qquad\qquad \land (\forall v \in \top B^*QB^*)(u \sqsubset v \implies u' \sqsubseteq v)\bigr\}\\
&\cup \{(w,w) : w \in L - \top B^*QB^*\}
\end{align*}
is regular since an automaton can recognize the $\sqsubset$ relation.

Thus $(L,\phi)$ is an FA-presentation for the algebra $\mathcal{X} =
(X,\circ,\alpha,\beta,\gamma,\delta,\zeta)$.

By \fullref{Lemma}{lem:algebrafg} below, the algebra $\mathcal{X}$ is
finitely generated. Let $g_0$ be the initial configuration of
$\tm{T}$ in the configuration graph $G_\bot$. Let $\mathcal{Y}$ be the
subalgebra generated by $g_0$. Then $\mathcal{Y}$ is not
FA-pre\-sent\-a\-ble by \fullref{Lemma}{lem:subalgebranotfa} below.
\end{example}

%% We concede that the example is an artificial
%% non-homogeneous combination of unrelated structures, and await a more
%% natural example.

\begin{lemma}
\label{lem:algebrafg}
The algebra $\mathcal{X}$ is finitely generated.
\end{lemma}

\begin{proof}
Let $u$ be the $\sqsubseteq$-minimal word in $\top B^*QB^*$. Let $T$ be the
set of elements in the subalgebra generated by $u\phi \in G_\top$; the
aim is to show that $T = X$.

By repeated application of the operation $\zeta$ to $u\phi$, all
elements of $G_\top$ lie in $T$. By applying $\alpha$ to elements of
$G_\top$, all elements of $G_\bot$ lie in $T$. By applying $\gamma$ to
those configurations in $G_\bot$ where the state is $q_\square$ and
the tape contains $a^k$ for some $k \in \nset$, the maximum elements
of each chain $M_i$ lie in $T$. By repeatedly applying $\beta$ to
these maximum elements, all elements of the chains $M_i$ lie in $T$,
as does $0_S$. Hence all elements of $X$ lie in $T$ and so
$X = T$.
\end{proof}

\begin{lemma}
\label{lem:subalgebranotfa}
The subalgebra $\mathcal{Y}$ is not FA-pre\-sent\-a\-ble.
\end{lemma}

\begin{proof}
The first step is to show that the subalgebra $\mathcal{Y}$ contains
the chains $M_{j^2}$ and no other chains $M_i$.

Recall that $\mathcal{Y}$ is generated by $g_0$, the initial
configuration of $\tm{T}$ in $G_\bot$. The operation $\delta$ applied
repeatedly to $g_0$ yields every element reachable from $g_0$ in
the configuration graph $G_\bot$. Let $H$ be the set of these
reachable elements. By the definition of $\tm{T}$, the set $H$ includes
configurations with state $q_\square$ and tape contents $a^{j^2}$ for
all $j \in \nset$. Furthermore, the definition of $\tm{T}$ ensures
that $H$ includes no other configuration with state $q_\square$. The
operation $\gamma$ applied to $H$ yields the maximum element of
every $M_{j^2}$ (where $j \in \nset$). The operation $\beta$ now
yields all elements of each $M_{j^2}$ and also yields $0_S$. So
$\mathcal{Y}$ contains the set $Y = \{0_S\} \cup H \cup \bigcup_{j \in \nset}
M_{j^2}$. It is easy to see that $Y$ is closed under
every operation. So the domain of $\mathcal{Y}$ is $Y$. In particular,
$\mathcal{Y}$ contains the chains $M_{j^2}$ and no other chains $M_i$.

Now suppose, with the aim of obtaining a contradiction, that
$\mathcal{Y}$ admits an injective automatic presentation
$(L,\phi)$. Let
\begin{equation}
\label{eq:multuse1}
Y_1 = \{y \in Y : y \circ 0_S = 0_S\};
\end{equation}
notice that $Y_1 = \{0_S\} \cup \bigcup_{j\in\nset} M_{j^2}$. Notice further that $Y_1$ is defined by a
first order formula and so $L_1 = Y_1\phi^{-1}$ is regular. Observe that
the order relation $\leq$ on the subsemilattice $Y_1$ is first-order
definable in terms of $\circ$. Let
\begin{equation}
\label{eq:multuse2}
K_1 = \bigl\{u \in L_1 : (\forall v)\bigl((u\phi \leq v\phi) \implies (u\phi = v\phi)\bigr)\bigr\}.
\end{equation}
Then $K_1$ consists of representatives in $L$ of the maximum elements
in the various sub-chains $M_{j^2}$ of $Y_1$. Since it is defined by a
first-order formula, $K_1$ is regular. Let
\[
K_2 = \bigl\{u \in K_1 : (\forall v \in L_1)\bigl((v \in K_1 \wedge |u| = |v|) \implies (v \sqsubseteq u)\bigr)\bigr\};
\]
then $K_2$ consists of length-plus-lexicographically minimal words of
each length in $K_1$. The language $K_2$ is regular. The
relation
\begin{equation}
\label{eq:multuse3}
R_1 = \bigl\{(u,v) : (u \in K_2) \wedge (u\phi \geq v\phi) \wedge (v\phi \neq 0_S)\bigr\}
\end{equation}
is regular. Notice that $R_1$ relates a word $u \in K_2$, which
represents the maximum element of some chain $M_{j^2}$, to all the
words $v$ representing elements of that chain. Let $n$ be the number of
states in an automaton recognizing $\conv(R_1)$.

If $(u,v) \in R_1$, then $|v| \leq |u| + n$, for otherwise one could
pump the subword of $v$ that extends beyond $u$ to obtain infinitely
many words representing elements of a single $M_{j^2}$, which would entail
infinitely many distinct elements of $M_{j^2}$ (since $\phi$ is
injective), which is a contradiction.

Let
\[
R_2 = \{(u\#^n,v) : (u,v) \in R_1\},
\]
where $\#$ is a new symbol. By the observation in the last paragraph,
if $(u,v) \in R_2$, then $|u| \geq |v|$. Furthermore, if
$(u,v),(u',v') \in R_2$ and $|u| = |u'|$, then $u = u'$ by the
definition of $R_2$ and $K_2$. Moreover, no word in $\conv(R_2)$ contains a
letter whose left-hand component is $\$$. Therefore the number of
words of length $k$ in $\conv(R_2)$ is either $0$ or, if there is a
word $u\in K_2$ of length $k-n$, the number of possible words $v$ such
that $(u\#^n,v)$ lies in $R_2$, which is in turn the number of
elements of the chain $M_{j^2}$ in which $u\phi$ lies, which is
$2^{j^2}$.

Let $z_k$ be the number of words in $\conv(R_2)$ of length $k$. By the
observation in the last paragraph, whenever $z_k$ is non-zero, it is
the number of elements in some chain $M_{j^2}$. Since $\conv(R_2)$ is a
regular language, the generating function
\[
f(x) = \sum_{k=0}^\infty z_kx^k
\]
is a rational function with no singularity at $0$. Thus the radius of
convergence of its power series expansion must be strictly greater
than zero. The aim is to obtain a contradiction by showing that this
power series has radius of convergence zero.

By the pumping lemma for regular languages, there are constants $p,q$
such that $z_{p+kq}$ is non-zero for all $k \in \nset^0$. So for every
$k \in \nset^0$, there exists $k\vartheta \in \nset^0$ such that
$z_{p+kq} = 2^{(k\vartheta)^2}$. This defines an injection $\vartheta
: \nset^0 \to \nset^0$. By \fullref{Lemma}{lem:injectiongeq} below, $k
\leq k\vartheta$ for infinitely many values of $k \in \nset^0$. So by
choosing $k$ to be large enough and also satisfying $k \leq
k\vartheta$, the value
\[
|z_{p+kq}|^{1/(p+kq)} = \Bigl|2^{(k\vartheta)^2}\Bigr|^{1/(p+kq)} = \Bigl|2^{(k\vartheta)^2/(p+kq)}\Bigr| 
\]
can be made arbitrarily large. Therefore
\[
\limsup_{k\to\infty} |z_{p+kq}|^{1/(p+kq)} = \infty,
\]
and hence $\limsup_{k\to\infty} |z_k|^{1/k} = \infty$, from which it follows that
the radius of convergence of the power series $\sum_{k=0}^\infty
z_kx^k$ is zero.
\end{proof}

\begin{lemma}
\label{lem:injectiongeq}
Let $\vartheta : \nset^0 \to \nset^0$ be an injection. Then there are
infinitely many $i \in \nset$ such that $i \leq i\vartheta$.
\end{lemma}

\begin{proof}
Suppose, with the aim of obtaining a contradiction, that there are
only finitely many $i \in \nset^0$ such that $i \leq i\vartheta$. Let
$I = \{i \in \nset : i \leq i\vartheta\}$; by supposition, $I$ is
finite. Let $m = \max(I)$ and $n = \max(I\vartheta)$. Then $m \leq
m\vartheta$ and $m\vartheta \leq n$, so $m \leq n$. Furthermore,
$i\vartheta < i$ for $i \notin I$, and $i\vartheta \leq n$ for $i \in
I$. Hence $i\vartheta \leq n$ for all $i \leq n$. Since $m = \max(I)$
and $m \leq n$, it follows that $n+1 \notin I$ and so $(n+1)\vartheta
< n+1$. Putting the last two sentences together shows that
$\{0,\ldots,n+1\}\vartheta \subseteq \{0,\ldots,n\}$, which
contradicts $\vartheta$ being an injection. Thus there are infinitely
many $i \in \nset^0$ such that $i\vartheta \geq i$.
\end{proof}

\begin{remark}
\label{rem:evenunary}
In \fullref{Example}{ex:nonfasubalgebra}, the algebra $\mathcal{X}$
has exactly one binary operation, namely the multiplication
$\circ$. However, this is not used for the finite generation of
$\mathcal{X}$ or $\mathcal{Y}$, and is used in only three places in
the proof of \fullref{Lemma}{lem:subalgebranotfa}, namely
\eqref{eq:multuse1}, \eqref{eq:multuse2}, and \eqref{eq:multuse3}. (In
the latter two, it is hidden within the first-order definition of the
order $\leq$ on $S$.)

However, we can modify $\mathcal{X}$ by removing $\circ$ and adding
two new unary operations $\lambda$ and $\mu$ to obtain a new algebra
$\mathcal{X}'$, where the subalgebra $\mathcal{Y}'$ generated by
$g_0$ has the same domain $Y$ as $\mathcal{Y}$ and where the
$\mathcal{Y}'$ can be proved to be non-FA-pre\-sent\-a\-ble in the same
way. Hence, even the class of FA-pre\-sent\-a\-ble algebras with only unary
operations is not closed under forming finitely generated subalgebras.

The first operation $\lambda$ sends every element of each chain $M_i$ to
the maximum element of that chain, and acts like the identity map
elsewhere (that is, on $\{0_S\} \cup G_\top \cup G_\bot$). The second
operation $\mu$ sends every element of $S$ to $0_S$ and acts like the
identity map elsewhere (on $G_\top \cup G_\bot$). Notice that
\begin{align*}
\Lambda(\lambda,\phi) &= \{(u,1^{|u|}) : u \in \{0,1\}^*\} \cup \{(u,u) : u \in \{z\} \cup \{\top,\bot\}B^*QB^*\} \\
\Lambda(\mu,\phi) &= \{(u,z) : u \in \{z\} \cup \{0,1\}^*\} \cup \{(u,u) : u \in \{\top,\bot\}B^*QB^*\};
\end{align*}
both relations are clearly regular. Thus $\mathcal{X}' =
(X,\alpha,\beta,\gamma,\delta,\zeta,\lambda,\mu)$ is FA-pre\-sent\-a\-ble and
has only unary operations.

To prove that the subalgebra $\mathcal{Y}'$ (generated by $g_0$) is
not FA-pre\-sent\-a\-ble, follow the proof of
\fullref{Lemma}{lem:subalgebranotfa}, with the following definitions
for $Y'$, $K_1$, and $R_1$ replacing \eqref{eq:multuse1},
\eqref{eq:multuse2}, and \eqref{eq:multuse3}:
\begin{align*}
Y_1 &= \{y \in Y : y\mu = 0_S\}, \\
K_1 &= \bigl\{u \in L_1 : ((u\phi)\lambda = u\phi) \land (u\phi \neq 0_S)\bigr\}, \\
R_1 &= \bigl\{(u,v) : (u \in K_2) \land (u\phi = (v\phi)\lambda)\bigr\}.
\end{align*}
Note that these are first order definitions in terms of the new
signature in which $\lambda$ and $\mu$ replace $\circ$.
\end{remark}

\section{Diagrams for unary FA-pre\-sent\-a\-tions}
\label{sec:pumping}

This section develops a diagrammatic representation for unary
FA-pre\-sent\-a\-tions. In the following section, we apply this
representation to prove results about subalgebras of unary
FA-pre\-sent\-a\-ble algebras.

Let $(a^*,\phi)$ be an injective unary FA-pre\-sent\-a\-tion for a
relational structure $\mathcal{S}$ with relations $R_1,\ldots,R_n$;
such an FA-pre\-sent\-a\-tion exists by \fullref{Theorem}{thm:allwords}. For
each $i \in \{1,\ldots,n\}$, let $\fsa{A}_i$ be a deterministic
$r_i$-tape automaton recognizing $\Lambda(R_i,\phi)$, where $r_i$ is
the arity of $R_i$. Let us examine the structure of the automata
$\fsa{A}_i$. For ease of explanation, view $\fsa{A}_i$ as a directed
graph with no failure states: $\fsa{A}_i$ fails if it is in a
state and reads a symbol that does not label any outgoing edge from
that state.

Define a partial order $\prec$ on elements of $\{a,\$\}^{r_i}$ by
$(x_1,\ldots,x_{r_1}) \preceq (x'_1,\ldots,x'_{r_1})$ if and only if
$x'_i = \$ \implies x_i = \$$ for all $i$. Since $\fsa{A}_i$
recognizes words in $\conv((a^*)^{r_i})$, it will only successfully
read words consisting of a $\preceq$-decreasing sequence of tuples in
$\{a,\$\}^{r_i}$. Thus an edge labelled by a tuple $b$ leads to a
state all of with outgoing edges are labelled by $\preceq$-preceding
tuples.

Since $\fsa{A}_i$ is deterministic, while it reads letters of a fixed
tuple $b \in \{a,\$\}^{r_i}$, it follows a fixed path which, if the
string of letters $b$ is long enough, will form a uniquely determined
loop. This loop, if it exists, is simple. From various points along
this loop and the path leading to it, paths labelled
by $\prec$-preceding letters of
$\{a,\$\}^{r_i}$ may `branch off'. \fullref{Figure}{fig:unaryautomaton} shows an
example where $r_i$ is $2$.

\begin{figure}[tb]
\centerline{\includegraphics{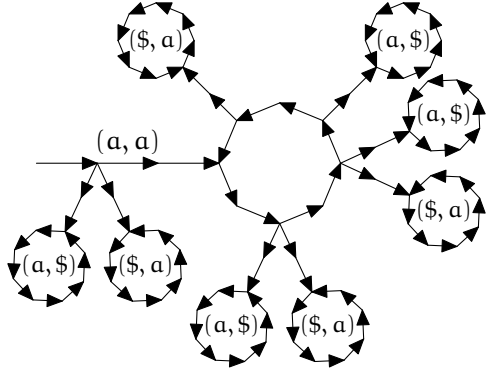}}
\caption{Example automaton recognizing $\Lambda(R_i,\phi)$ where $r_i$
  is $2$. Edges labelled $(a,a)$ form a path that leads into a
  uniquely determined loop. From this path and loop paths labelled by
  $(a,\$)$ or $(\$,a)$ branch off. (Notice that $(a,\$), (\$,a) \prec
  (a,a)$.)}
\label{fig:unaryautomaton}
\end{figure}

Let $D_i$ be a multiple of the lengths of the loops in $\fsa{A}_i$ (as
discussed above) that also exceeds the number of states in
$\fsa{A}_i$. Let $D$ be a multiple of the various $D_i$.

Fix some $i$. Let $\fsa{A}_i$ have initial state $q_0$ and transition
function $\delta$. Consider a word $uvw \in \conv(L(\fsa{A}_i))$, where
$v = b^\beta$ for some $b \in (A \cup \{\$\})^{r_i}$ and $\beta \geq
D$. Suppose that $(q_0,u)\delta = q$. When $\fsa{A}_i$ is in state $q$
and reads $v$, it completes a loop before finishing reading $v$. (By
the discussion above, the loop is simple and uniquely determined.) So
$v$ factorizes as $v'v''v'''$, with $|v''| > 0$, such that
$(q,v')\delta = (q,v'v'')\delta = q'$. Assume that $|v'|$ is minimal,
so that $q'$ is the first state on the loop that $\fsa{A}_i$
encounters while reading $v$. Assume further that $|v''|$ is minimal,
so that $\fsa{A}_i$ makes exactly one circuit around the loop while
reading $v''$. Now, by definition, $D$ is a multiple of $|v''|$. Let
$m = D/|v''|$. So $|v(v'')^{m+1}v'''| = |v| + D$. By the pumping
lemma, $uv'(v'')^{m+1}v'''w \in \conv(L(\fsa{A}_i))$.

Consider what this means in terms of the tuple $\vec{p} =
(a^{p_1},\ldots,a^{p_{r_i}})$ such that $\conv(\vec{p}) = uvw$. Since $v' \in
b^*$, it follows that
\[
uv'(v'')^{m+1}v'''w = \conv(a^{p_1+q_1},\ldots,a^{p_{r_i}+q_{r_i}}),
\]
where
\[
q_j = \begin{cases} 0 & \text{if $p_j \leq |u|$} \\
D & \text{if $p_j \geq |uv|$}.
\end{cases}
\]
(Note that either $p_j \leq |u|$ or $p_j \geq |uv|$ since $v \in b^*$
for a fixed $b \in (A \cup \{\$\})^{r_i}$.) Therefore we have the following:

\begin{pumprule}
\label{rule:up}
If the components of a tuple in $\Lambda(R_i,\phi)$ can be
partitioned into those that are of length at most $l \in
\nset$ and those that have length at least $l + D$, then [the word
encoding] this tuple can be pumped so as to increase by $D$ the
lengths of those components that are at least $l + D$ letters
long and yield another [word encoding a] tuple in $\Lambda(R_i,\phi)$.
\end{pumprule}

(Notice that this also applies when all components have length
at least $D$; in this case, set $l = 0$.) 

With the same setup as above, suppose $|v| \geq 2D$. Then $\fsa{A}_i$
must follow the loop labelled by $v''$ starting at $q'$ at least
$m = D/|v''|$ times. That is, $v$ factorizes as
$v'(v'')^m\tilde{v}'''$. By the pumping lemma, $uv'\tilde{v}''' \in
\conv(L(\fsa{A}_i))$ and $|v'\tilde{v}'''| = |v| - D$. Therefore, we also
have the following:

\begin{pumprule}
\label{rule:down}
If the components of a tuple in $\Lambda(R_i,\phi)$ can be divided
into those that are of length less than $l \in \nset$ and those
that have length at least $l + 2D$, then [the word encoding] this
tuple can be pumped so as to decrease by $D$ the length of those
components that are at least $l + 2D$ letters long and yield
another [word encoding a] tuple in $\Lambda(R_i,\phi)$.
\end{pumprule}

This ability to pump so as to increase or decrease lengths of
components by a constant $D$ lends itself to a very useful
diagrammatic representation of the unary FA-pre\-sent\-a\-tion
$(a^*,\phi)$. Consider a grid of $D$ rows and infinitely many
columns. The rows, from bottom to top, are $B[0],\ldots,B[D-1]$. The
columns, starting from the left, are $C[0],C[1],\ldots$. The point in
column $C[x]$ and row $B[y]$ corresponds to the word $a^{xD+y}$. For
example, in the following diagram, the distinguished point is in
column $C[3]$ and row $B[2]$ and so corresponds to $a^{3D+2}$:

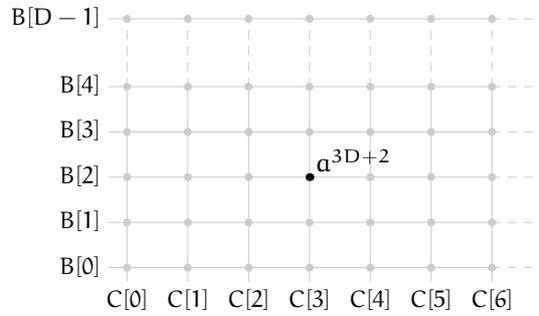
\inpargraphic{\jobname-pumpingdiagram1.eps}

The power of such diagrams is due to a natural correspondence between
pumping as in \fullref{Pumping rules}{rule:up} and~\ref{rule:down} and
certain simple manipulations of tuples of points in the
diagram. Before describing this correspondence, we must set up some
notation. We will not distinguish between a point in the grid and the
word to which it corresponds. The columns are ordered in the obvious
way, with $C[x] < C[x']$ if and only if $x < x'$. Extend the notation
for intervals on $\nset$ to intervals of contiguous columns. For
example, for $x,x' \in \nset$ with $x \leq x'$, let $C[x,x')$ denotes
  the set of elements in columns $C[x],\ldots,C[x'-1]$, and
  $C(x,\infty)$ denotes the set of elements in columns
  $C[x+1],C[x+2],\ldots$. For any element $u \in a^*$, let $c(u)$ be
  the index of the column containing $u$.

Consider the components of an $r_i$-tuple $\vec{p}$ in
$\Lambda(R_i,\phi)$, viewed as an $r_i$-tuple of points in the
diagram. If there is a column $C[x]$ that contains none of the
components of $\vec{p}$, then all the components that lie in $C[0,x)$
  are at least $D$ shorter than those lying in $C(x,\infty)$. Hence
  the word encoding the tuple $\vec{p}$ can be pumped between these
  two sets of components in accordance with \fullref{Pumping
    rule}{rule:up}. This corresponds to shifting all those components
  that lie in $C(x,\infty)$ rightwards by one column. The tuple that
  results after this rightward shift of some components also lies in
  $\Lambda(R_i,\phi)$.  [Notice in particular that if column $C[0]$
    contains none of the components of $\vec{p}$, then every component
    can be shifted right by one column, giving a new tuple that also
    lies in $\Lambda(R_i,\phi)$.] This rightward shifting of
  components can be iterated arbitrarily many times to yield new
  tuples. Thus we have the following diagrammatic version of
  \fullref{Pumping rule}{rule:up}:

\begin{shiftrule}
\label{shiftrule:up}
Consider the components of an $r_i$-tuple $\vec{p}$ in
$\Lambda(R_i,\phi)$, viewed as an $r_i$-tuple of points in the
diagram. If there is a column $C[x]$ that contains none of the
components of $\vec{p}$, then for any $k \in \nset$, shifting the
components in $C(x,\infty)$ to the right by $k$ columns yields a
tuple that also lies in $\Lambda(R_i,\phi)$.
\end{shiftrule}

Similarly, if there are two adjacent columns $C[x]$ and $C[x+1]$ that
contain none of the components of $\vec{p}$, then every component in
$C[0,x)$ is at least $2D$ shorter than every component in
  $C(x+1,\infty)$. Therefore the word encoding this tuple can be
  pumped between these sets of components in accordance with
  \fullref{Pumping rule}{rule:down}. This corresponds to shifting all
  components in $C(x+1,\infty)$ leftwards by one column. The tuple
  that results after this leftward shift of some components also lies
  in $\Lambda(R_i,\phi)$. [Notice in particular that if columns $C[0]$
    and $C[1]$ contain none of the components of $\vec{p}$, then every
    component can be shifted left by one column, giving a new tuple
    that also lies in $\Lambda(R_i,\phi)$.] This leftward shifting of
  components can be iterated to yield new tuples for as long as the
  two columns $C[x]$ and $C[x+1]$ do not contain any elements of the
  latest tuple. Thus we have the following diagrammatic version of
  \fullref{Pumping rule}{rule:down}:

\begin{shiftrule}
\label{shiftrule:down}
Consider the components of an $r_i$-tuple $\vec{p}$ in
$\Lambda(R_i,\phi)$, viewed as an $r_i$-tuple of points in the
diagram. If the columns in $C[x,x+h]$ contain none of the
components of $\vec{p}$, then for any $k$ with $0 < k \leq h$,
shifting the components in $C(x+h,\infty)$ to the left by $k$
columns yields a tuple that also lies in $\Lambda(R_i,\phi)$.
\end{shiftrule}

For convenience, define for every $n \in \zset$ a partial map $\tau_n
: a^* \to a^*$, where $a^k\tau_n$ is defined to be $a^{k+nD}$ if $k+nD
\geq 0$ and is otherwise undefined. Notice that if $n \geq 0$, the map
$\tau_n$ is defined everywhere. In terms of the diagram, $a^k\tau_n$
is the element obtained by shifting $a^k$ to the right by $n$ columns
if $n \geq 0$ and to the left by $-n$ columns if $n < 0$. The values
of $k$ and $n < 0$ for which $a^k\tau_n$ are undefined are precisely
those where shifting $a^k$ to the left by $-n$ columns would carry it
beyond the left-hand edge of the diagram.

\begin{example}
In order to illustrate \fullref{Shift rules}{shiftrule:up} and
\ref{shiftrule:down}, consider a $4$-tuple $\vec{p} =
(a^{2D+1},a^{D+2},a^{5D+3},a^{7D+2})$. This corresponds to the
following four points in the diagram.

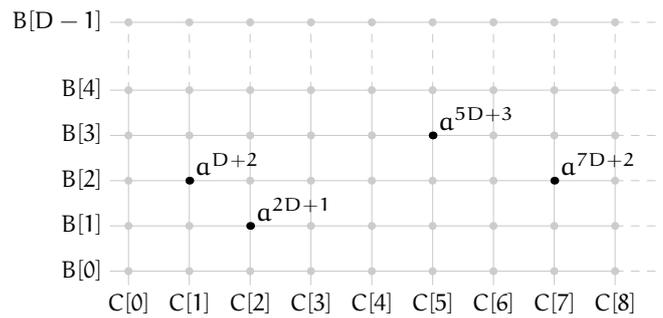
\inpargraphic{\jobname-pumpingdiagram-example1.eps}

\fullref{Shift rule}{shiftrule:up} (or the corresponding
\fullref{Pumping rule}{rule:up}) can be applied in exactly three
ways here:
\begin{enumerate}

\item The column $C[0]$ contains no components of $\vec{p}$, so, by
  \fullref{Shift rule}{shiftrule:up}, for any $k \in \nset$, all
  components can be shifted rightward by $k$ columns,
  yielding the tuple $(a^{(2+k)D+1},a^{(1+k)D+2},a^{(5+k)D+3},a^{(7+k)D+2})$.

\item The columns $C[3]$ and $C[4]$ contain no components of
  $\vec{p}$, so, by \fullref{Shift rule}{shiftrule:up}, for any $k \in
  \nset$, the third and fourth components can be shifted rightwards by
  $k$ columns, yielding the tuple
  $(a^{2D+1},a^{D+2},a^{(5+k)D+3},a^{(7+k)D+2})$.

\item The column $C[6]$ contains no components of
  $\vec{p}$, so, by \fullref{Shift rule}{shiftrule:up}, for any $k \in
  \nset$, the fourth component of $\vec{p}$ can be shifted rightwards by
  $k$ columns, yielding the tuple
  $(a^{2D+1},a^{D+2},a^{5D+3},a^{(7+k)D+2})$.

\end{enumerate}

\fullref{Shift rule}{shiftrule:down} (or the corresponding
\fullref{Pumping rule}{rule:down}) can be applied in only one way here:
columns $C[3]$ and $C[4]$ contain no component of $\vec{p}$, so
the third and fourth components can be shifted leftwards by one column,
yielding the tuple $(a^{2D+1},a^{D+2},a^{4D+3},a^{6D+2})$.
\end{example}

\section{Unary FA-pre\-sent\-a\-ble algebras}

This section studies finitely generated subalgebras of unary
FA-pre\-sent\-a\-ble algebras. The key result,
\fullref{Theorem}{thm:subalgebralang}, shows that the language
representing elements of such a subalgebra is regular and that there
is an algorithm that effectively constructs this language. From this
it follows that the class of unary FA-pre\-sent\-a\-ble algebras is closed
under taking finitely generated subalgebras
(\fullref{Theorem}{thm:subalgebras}) and that the membership problem
for finitely generated subalgebras is decidable
(\fullref{Theorem}{thm:membership}).

\begin{theorem}
\label{thm:subalgebralang}
Let $\mathcal{S}$ be an algebra that admits an injective unary
FA-pre\-sent\-a\-tion $(a^*,\phi)$ and let $\mathcal{T}$ be a finitely
generated subalgebra of $\mathcal{S}$. Let $L$ be the sublanguage of
$a^*$ consisting of representatives of elements of $\mathcal{T}$. Then
$L$ is regular, and an automaton recognizing $L$ can be constructed
effectively from a finite set of words representing a generating set
for $\mathcal{T}$.
\end{theorem}

\begin{proof}
We will first of all show $L$ is regular and then show how it can be
constructed effectively.

\medskip
\noindent\textit{Regularity.} Let $r$ be the maximum arity of any of
the operations in the signature of $\mathcal{S}$. (That is, $r$ is the
maximum of their arities \textit{qua} operations, not \textit{qua}
relations.)

For $x,y \in \nset^0$ with $x \leq y$, define new notation $L[x] = C[x]
\cap L$ and $L[x,y] = C[x,y] \cap L$.  Since the subalgebra
$\mathcal{T}$ is finitely generated, it is generated by the elements in
$L[0,m]\phi$ for some $m \in \nset^0$.

Suppose that $h,h' \in \nset \cup \{0\}$, where
$h' > h > m$, are such that
\begin{equation}
\label{eq:subalgebra1}
L[h,h+r]\tau_{h'-h} \subseteq L[h',h'+r].
\end{equation}
The immediate aim is to prove that
\begin{equation}
\label{eq:subalgebra2}
L[h+r+1]\tau_{h'-h} \subseteq L[h'+r+1].
\end{equation}

Because $h+r > m$, the elements of $L[0,h+r]\phi$ generate the
elements of $L[h+r+1]\phi$. That is, by applying the operations of
$\mathcal{S}$ to elements of $L[0,h+r]\phi$, one can obtain a finite
sequence of points $p_1,\ldots,p_n \in L[h+r+1,\infty)$ such that each
  $p_i\phi$ is obtained by a single application of some operation to
  elements from $(L[0,h+r] \cup \{p_1,\ldots,p_{i-1}\})\phi$, and such
  that $L[h+r+1]\phi \subseteq \{p_1,\ldots,p_n\}\phi$. [It may be necessary
    for some $p_i$ to lie in columns to the right of $C[h+r+1]$, in
    order to later generate the elements of $L[h+r+1]\phi$.]

\begin{figure}[tb]
\centerline{\includegraphics{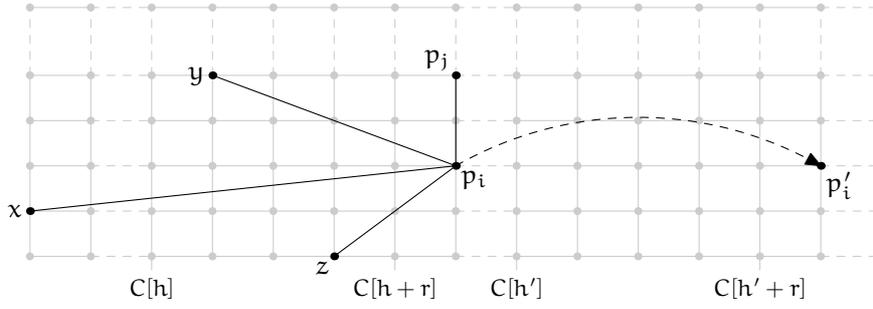}}
\caption{The indices $h$ and $h'$ are such that an elements of
  $L[h,h+r]\tau_{h'-h} \subseteq L[h',h'+r]$. The point
  $p'_i$ is defined to be $p_i\tau_{h'-h}$. The solid lines indicate
  how $p_i\phi$ is obtained by an application of some operation to the
  elements $x\phi$, $y\phi$, $z\phi$, and $p_j\phi$ (where $j < i$).}
\label{fig:subalgebra1}
\end{figure}

For each $i \in \{1,\ldots,n\}$, let $p'_i = p_i\tau_{h'-h}$. The
aim is to prove by induction on $i$ that $p'_i \in L$. We
will show that, just as $p_i\phi$ is obtained by an application of some
operation to elements from $(L[0,h+r] \cup
\{p_1,\ldots,p_{i-1}\})\phi$, so $p'_i\phi$ can be obtained by an application
of the same operation to elements of
\[
\bigl(L[0,h+r] \cup L[h',h'+r] \cup \{p'_1,\ldots,p'_{i-1}\}\tau_{h'-h}\bigr)\phi.
\]

So suppose that $p'_1,\ldots,p'_{i-1} \in L$. Suppose $p_i\phi =
(x_1\phi,\ldots,x_k\phi)f$, where $f$ is an operation of arity $k \leq
r$ and $x_1,\ldots,x_k \in L[0,h+r] \cup \{p_1,\ldots,p_{i-1}\}$. (See
\fullref{Figure}{fig:subalgebra1}.)  Without loss of generality,
assume that $c(x_j) \leq c(x_{j+1})$ for all $j \in
\{1,\ldots,k-1\}$. Since $k \leq r$, there is at least one column $C'$
in $C[h],\ldots,C[h+r]$ that does not contain any point
$x_1,\ldots,x_k$. Let $x_1,\ldots,x_j$ be the points lying to the left
of this column, and $x_{j+1},\ldots,x_k$ be those lying to the right.

\begin{figure}[tb]
\centerline{\includegraphics{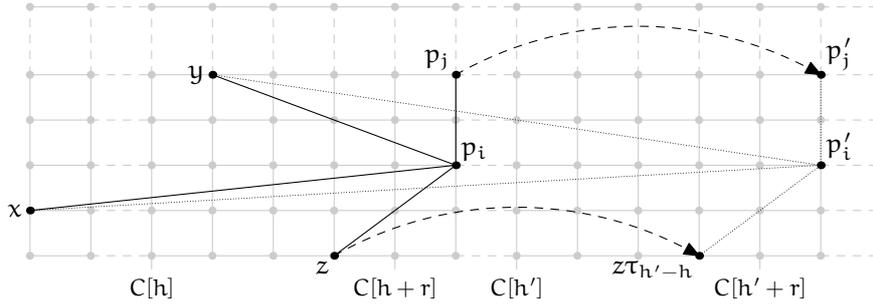}}
\caption{The dotted lines indicate how $p'_i$ is obtained by an
  application of the same operation to the elements $x$, $y$,
  $z\tau_{h'-h}$, and $p'_j = p_j\tau_{h'-h}$.}
\label{fig:subalgebra2}
\end{figure}

For $l \in \{j+1,\ldots,k\}$, let $x'_l = x_l\tau_{h'-h}$. Recall that
$p'_i = p_i\tau_{h'-h}$. Now, since $x_l \in L[h,h+r]$, it
follows that $x'_l \in L[h',h'+r]$ by \eqref{eq:subalgebra1}. On
the other hand, if $x'_l$ is one of the points $p'_1,\ldots,p'_{i-1}$,
then it lies in $L$ by the induction hypothesis. The application of
the operation $f$ to the elements $x_1\phi,\ldots,x_k\phi$ gives $p_i\phi$. Let
$\vec{p} = (x_1,\ldots,x_k,p_i) \in \Lambda(f,\phi)$. Then by \fullref{Shift
  rule}{shiftrule:up}, the tuple
\[
\vec{q} = (x_1,\ldots,x_j,x'_{j+1},\ldots,x'_k,p'_i),
\]
obtained by shifting rightwards the components $x_{j+1},\ldots,x_k,p_i$,
also lies in $\Lambda(f,\phi)$. Since all of $x_1\phi,\ldots,x_j\phi$ and
$x'_{j+1}\phi,\ldots,x'_k\phi$ lie in the subalgebra $\mathcal{T}$, so does $p'_i\phi$. (See
\fullref{Figure}{fig:subalgebra2}.) Hence $p'_i \in L$.

Therefore, by induction, all the points $p'_i$ lie in $L$, and hence
condition~\eqref{eq:subalgebra2} holds. Thus
condition~\eqref{eq:subalgebra1} entails
condition~\eqref{eq:subalgebra2}.

Since each of the sets $L[h,h+r]$ contains at most $(r+1)D$ elements,
there must exist $h,h' \in \nset \cup \{0\}$ with $h'>h$ such that
\eqref{eq:subalgebra1} holds. Fix two such values $h$ and $h'$. Then
it follows by induction on $i$ that $L[i]\tau_{h'-h} \subseteq
L[i+h'-h]$ for all $i \geq h$. Since the size of the sets $L[i]$ is
bounded above by $D$, there exists $g \in \nset$ such that
$L[i]\tau_{h'-h} = L[i+h'-h]$ for all $i \geq g$. Thus
\[
L = L[0,g-1] \cup (a^{D(h'-h)})^*L[g,g+h'-h-1]
\]
and so is regular.

\medskip
\noindent\textit{Effective construction.} Let $L_0$ be a finite set of
words representing a generating set for the subalgebra
$\mathcal{T}$. We will inductively construct a sequence of regular
sublanguages $L_i$ of $L$ for $i \in \nset^0$. From some point
onwards, every language in this sequence will be $L$ itself. We will
be able to detect when $L_i = L$, but we cannot bound in advance the
number of terms we must compute before obtaining $L$. For all $i,x,y
\in \nset^0$ with $x \leq y$, let $L_i[x] = C[x] \cap L_i$ and $L_i[x,y] = C[x,y] \cap
L_i$.

Inductively define the language $L_{i+1}$ as follows: find the
minimal $h$ such that there exists $h'$ such that
\[
L_i[h,h+r]\tau_{h'-h} \subseteq L_i[h',h'+r].
\]
(Notice that this is \eqref{eq:subalgebra1} restated with $L_i$ in
place of $L$.) Let $h_i$ be $h$ and let $h'_i$ be minimal among
corresponding such $h'$, and let
\[
L_{i+1} = L_{i} \cup \bigl((a^{D(h'_i-h_i)})^*L_i[h_i,h'_i-1]\bigr) \cup K_{i+1}
\]
where
\[
K_{i+1} = \bigl\{(s_1,\ldots,s_{r_j},x)f_j : j \in \{1,\ldots,k\}, s_1,\ldots,s_{r_j} \in L_{i}\phi, 
\text{$f_j$ has arity $r_j$}\bigr\}\phi^{-1}.
\]

[Notice that $h_i$ and $h'_i$ always exist since because the sets
  $L_i[h_i,h'_i+r]$ are all of bounded size. Notice further that when
  $L_i$ is finite, $L_i[h_i,h'_i+r]$ may be empty. Observe that $h_i$
  and $h'_i$ can be found simply by enumerating sets $L_i[h,h+r]$.]

Let us prove by induction that $L_i \subseteq L$. Clearly $L_0
\subseteq L$.  Suppose that $L_i \subseteq L$. By the reasoning in
used in the proof of regularity above, each element of
$(a^{D(h'_i-h_i)})^*L_i[h_i,h'_i-1]$ lies in $L$. The language
$K_{i+1}$ consists of representatives of elements obtained by applying
the operations of $\mathcal{S}$ to elements of $L_i\phi$. Since
$\mathcal{T}$ is a subalgebra, every element of the language $K_{i+1}$
thus lies in $L$. Hence $L_{i+1} \subseteq L$.

Furthermore, the language $K_{i+1}$ consists of representatives of
elements satisfying a first-order formula. Hence, if $L_i$ is regular
and given by a finite automaton, a finite automaton recognizing
$L_{i+1}$ can be effectively constructed. Since $L_0$ is finite, it
follows by induction that every $L_i$ is regular, and that for any $i
\in \nset^0$ an automaton recognizing $L_i$ can be effectively
constructed.

Notice further that $L_i \subseteq L_{i+1}$ and that for any $u \in
L$, there exists some $L_i$ such that $u \in L_i$.

By the reasoning in the proof of regularity above, there exist $g,h,h'
\in \nset$ be such that $L[i]\tau_{h'-h} = L[i+h'-h]$ for all $i \geq
g$. Note that $L = L[0,g-1] \cup (a^{D(h'-h)})^*L[g,g+h'-h]$. Let $n$ be
such that $L[0,g+h'-h] \subseteq L_n$. Then, by definition, $L_{n+1}$
contains $L_n$ and $(a^{D(h'-h)})^*L[g,g+h'-h]$. Hence $L \subseteq
L_{n+1}$.

Therefore, the algorithm constructing the various $L_i$ will at some
point construct $L_{n+1} = L$. Furthermore, the algorithm can check
whether $L_i$ is $L$ simply by checking whether $L_i = L_{i+1}$, for
if this holds, then $L_i\phi$ is closed under all the operations of
$\mathcal{S}$ and hence must be the domain of the subalgebra
$\mathcal{T}$. Thus there is an effective procedure that constructs
$L$.
\end{proof}

\begin{theorem}
\label{thm:subalgebras}
The class of unary FA-pre\-sent\-a\-ble algebras is closed under taking
finitely presented subalgebras.
\end{theorem}

\begin{proof}
Let $\mathcal{S}$ be an algebra that admits an injective unary
FA-pre\-sent\-a\-tion $(a^*,\phi)$ and let $\mathcal{T}$ be a finitely
generated subalgebra of $\mathcal{S}$. Let $L$ be the sublanguage of
$a^*$ consisting of representatives of elements of $\mathcal{T}$. By
\fullref{Theorem}{thm:subalgebralang}, $L$ is regular, whence
\[
\Lambda(R,\phi|_L) = \Lambda(R,\phi) \cap \bigl(\underbrace{L \times L \times \ldots \times L}_{\text{$k$ times}}\bigr)
\]
for any $k$-ary relation (or operation) $R$ of $\mathcal{S}$, which
shows that $(L,\phi|_L)$ is a unary FA-pre\-sent\-a\-tion for $\mathcal{T}$.
\end{proof}

Note that \fullref{Theorem}{thm:subalgebras} does not hold without
the hypothesis of finite generation: an arbitrary subsemigroup of a
unary FA-pre\-sent\-a\-ble semigroup may not even be FA-pre\-sent\-a\-ble
\cite[Example~9.4]{crt_unaryfa}.

The following theorem deals with the membership problem for finitely
generated subalgebras of unary FA-presentable algebras. This problem
is not decidable for general FA-presentable algebras, because
reachability in the configuration graph of a Turing machine is
undecidable.

\begin{theorem}
\label{thm:membership}
There is an algorithm that takes a unary FA-pre\-sent\-a\-tion $(a^*,\phi)$
for an algebra $\mathcal{S}$, a finite set $X$ of words in $a^*$, and a
word $w \in a^*$, and decides whether $w\phi$ lies in the subalgebra
generated by $X\phi$.
\end{theorem}

\begin{proof}
By \fullref{Theorem}{thm:subalgebralang}, there is an algorithm that
takes the finite set of words $X$ and constructs the sublanguage $L$
of $a^*$ consisting of representatives of elements of the subalgebra
$\mathcal{T}$ generated by $X\phi$.  To decide whether $w\phi$ lies in
$\mathcal{T}$, it remains to check whether $w$ lies in $L$.
\end{proof}

Let $\mathcal{S}$ be a finitely generated algebra. Let
$f_1,\ldots,f_k$ be the operations of $\mathcal{S}$. Let $G_0$
be a finite generating set for $\mathcal{S}$. Inductively define the
following finite sets for all $i \in \nset$:
\[
G_i = G_{i-1} \cup \bigl\{(s_1,\ldots,s_{r_j})f_j : j \in \{1,\ldots,k\}, s_1,\ldots,s_{r_j} \in G_{i-1}, \text{$f_j$ has arity $r_j$}\bigr\}.
\]
Define $g : \nset^0 \to \nset^0$ by $n \mapsto |G_{n}|$. The function
$g$ is called the \defterm{growth level} of $\mathcal{S}$ with respect
to the generating set $G_0$. (This definition is taken from
\cite[\S~4]{khoussainov_autopres}.)

If $\mathcal{S}$ is FA-pre\-sent\-a\-ble, then there exist constants $s,a,b
\in \nset$ such that $g(n) \leq s^{a+1+bn}$ for all $n \in \nset^0$
\cite[Lemma~4.5]{khoussainov_autopres}.

\begin{proposition}
\label{prop:unaryalggrowth}
If $\mathcal{S}$ is a unary FA-pre\-sent\-a\-ble algebra, then there exist
constants $a,b \in \nset$ such that $g(n) \leq a+bn$ for all $n \in \nset^0$.
\end{proposition}

\begin{proof}
Let $(a^*,\phi)$ be an injective unary FA-pre\-sent\-a\-tion for
$\mathcal{A}$. Let $x \in \nset^0$ be such that $G_0 \subseteq
C[0,x]\phi$. The first aim is to prove, by induction, that $G_n
\subseteq C[0,x+n]\phi$ for all $n \in \nset^0$. This clearly
holds for $n=0$.

Let $f_j$ be an operation of $\mathcal{S}$ whose arity is $r_j$. Let
$u_1,\ldots,u_{r_j} \in G_i\phi^{-1} \subseteq C[0,x+n]$. Let $v \in
a^*$ be such that $(u_1\phi,\ldots,u_{r_j}\phi)f_j = v\phi$. Suppose
for \textit{reductio ad absurdum} that $c(v) > x+n+1$. Then no
component of the tuple $(u_1,\ldots,u_{r_j},v) \in \Lambda(f_j,\phi)$
lies in $C[x+n+1]$ and $v$ is the only component lying in
$C[x+n+2,\infty)$ and so, by \fullref{Shift rule}{shiftrule:up},
  $(u_1,\ldots,u_{r_j},v\tau_1) \in \Lambda(f_j,\phi)$. Hence $v\phi =
  (u_1\phi,\ldots,u_{r_j}\phi)f_j = (v\tau_1)\phi$, which contradicts
  the injectivity of $\phi$. Therefore $v \in C[0,x+n+1]$. Since
  $v\phi$ is the result of applying an arbitrary operation $f_j$ of
  $\mathcal{S}$ to arbitrary elements of $G_n$, it follows that
  $G_{n+1} \subseteq C[0,x+n+1]\phi$.

Let $a = (x+1)D$; then $|C[0,x]| = a$. Furthermore,
$|C[0,x+n+1]| = |C[0,x+n]| + D$. Let $b = D$; then
$|C[0,x+n+1]| = a + bn$. Since $G_n \subseteq
C[0,x+n+1]\phi$, it follows that $g(n) = |G_n| \leq a + bn$.
\end{proof}

The constrast between the growth levels of finitely generated
FA-pre\-sent\-a\-ble algebras (bounded by an exponential function) and
finitely generated \emph{unary} FA-pre\-sent\-a\-ble algebras (bounded
by a linear function) resembles the contrast between the growth of
finitely generated FA-pre\-sent\-a\-ble semigroups (polynomial growth
\cite[Theorem~7.4]{cort_apsg}) and finitely generated unary
FA-pre\-sent\-a\-ble semigroups (sublinear growth, which implies
finiteness \cite[Proof of Theorem~13]{crt_unaryfa}). Note, however,
the difference between the two types of growth: the growth level of an
algebra counts elements of a given term complexity, and the growth of
a semigroup or group counts elements of given word length.

%% \section*{Acknowledgements}

%% The first author's research was funded by the European Regional
%% Development Fund through the programme {\sc COMPETE} and by the
%% Portuguese Government through the {\sc FCT} (Funda\c{c}\~{a}o para a
%% Ci\^{e}ncia e a Tecnologia) under the project {\sc PEst-C}/{\sc
%%   MAT}/{\sc UI0}144/2011 and through an {\sc FCT} Ci\^{e}ncia 2008
%% fellowship. Two visits by the first author to the University of St
%% Andrews, where much of the research described in this paper was
%% carried out, were supported by the {\sc EPSRC}-funded project {\sc
%%   EP}/{\sc H0}11978/1 `Automata, Languages, Decidability in Algebra'.

%\bibliography{c_publications,automaticpresentations,languages,semigroups,\jobname}
\bibliography{\jobname_ext}
\bibliographystyle{alphaabbrv}

\end{document}